\documentclass{article}

\usepackage{graphicx} 
\usepackage{tikz} 
\usepackage{pdfpages} 
\usepackage{amsmath} 
\usepackage{mathtools} 
\usepackage{enumerate} 
\usepackage{amscd} 
\usepackage{stmaryrd} 
\usepackage{amsfonts} 
\usepackage{amsthm} 
\usepackage{amssymb} 
\usepackage{hyperref} 
\usepackage{ragged2e} 
\usepackage[utf8]{inputenc} 
\usepackage{float} 
\usepackage{enumitem} 
\newtheorem{theorem}{Theorem}
 
\newtheorem{lemma}[theorem]{Lemma} 
\newtheorem{prop}{Proposition} 
  
\title{\textbf{Klein Quartic Curve and its Modularity}} 
\author{PARESH SINGH ARORA} 
\date{} 
  
\begin{document} 
  
\maketitle 
  
\textsc{Abstract}: 
The local Zeta function of a variety encodes important information about the variety. From the works of Weil, Deligne, Dwork, and others, many things are known about the local Zeta function of a smooth projective variety. 
In this article, we find the local Zeta function for the Klein Quartic curve, $x^3y+y^3z+z^3x=0$, by realizing it as a quotient of degree 7 Fermat curve. We conclude by giving the associated modular forms via Galois representations.

\section{Introduction} 
  
The Klein Quartic curve is a genus three projective curve over $\mathbb{C}$, given by the equation $x^3y+y^3z+z^3x=0$. One easily observes that it is invariant under cyclic permutations of $(x,y,z)$. It has the highest possible number of symmetries according to Hurwitz's automorphism theorem \cite{Hurwitz}. The goal of this article is to compute the local Zeta function (and hence the global $L$-function) of the Klein Quartic curve and represent it in terms of modular forms.
  
\paragraph{} The global $L$-function of a variety $X$ over $\mathbb{Q}$, is defined by combining the local Zeta functions at all good primes, and the local Zeta function is constructed using the point-counting on $X$ over finite fields. 
For a prime number $p$, the quantity $N_{p^r}(X)$ represents the number of points on $X$ defined over the finite field $\mathbb{F}_{p^r}$. The local Zeta function of the variety $X$ reduced over $\mathbb{F}_p$, for a good prime $p$, is defined for a variable $
T$ as follows
\begin{equation}\label{zeta function of a variety} 
    Z_p(X;T)=\exp\left(\sum_{r=1}^{\infty}\frac{N_{p^r}(X)}{r} T^r\right), 
\end{equation}  
where $\displaystyle \exp(T)=\sum_{k=0}^{\infty}\frac{T^k}{k!}$. The Klein Quartic curve is the defining equation of the Modular curve $X(\Gamma(7)$) (for example see  \cite{YANG}), and hence its $L$-function is the product of $L$-functions of all three Hecke eigenforms on $\Gamma(7)$, that also happen to be CM forms. In this expository article, we establish the same fact by realizing the Klein Quartic curve as a quotient of degree-$7$ Fermat curve and looking through the perspective of Galois representations. 

\paragraph{} 
 The Galois representation of smooth projective varieties is often connected to automorphic forms. For example, elliptic curves over rationals are related to weight $2$ modular forms, and their relation is given by the much-celebrated Modularity theorem \cite[Chapter~8]{Diamond-Shurman}. The Klein Quartic curve is a genus three curve and has a $6$-dimensional Galois representation associated with it, by the absolute Galois group $G_{\mathbb{Q}}=Gal(\overline{\mathbb{Q}}/\mathbb{Q})$ acting on its first \'etale cohomology group. It is natural to ask if there are any modular forms associated with this Galois representation. If this is the case, then we have a way to extend the associated $L$-function analytically over the complex plane. This inquiry, in turn, leads us to the following theorem. 
\begin{theorem}\label{theorem 1} 
Let $S=$ \{newforms of level 49, weight 2 and admitting complex multiplication by the field $\mathbb{Q}(\sqrt{-7})\}$, where $|S|=3$. For $f\in S$, let $\chi_f$ be its corresponding character, then for a prime $p\neq 7$, 
\begin{equation}\label{MF equation1} 
Z_p(KQ;T)=\frac{\prod_{f\in S} (1-a_p(f)T+\chi_f(p)\cdot pT^2)}{(1-T)(1-pT)}. 
\end{equation} 
 Here, characters $\chi_f$ are Dirichlet characters modulo $49$ of order $1$ or $3$.
\end{theorem} 
  
\section{Preliminaries}  
\subsection{Gauss and Jacobi sums}\label{section Gauss sums}
We briefly introduce Gauss and Jacobi sums over finite fields and list the properties that we will use for counting points on curves. For a more detailed discussion of the topic, the reader can refer to \cite[Chapter~8]{Ireland-Rosen}.\\ 
Let $\mathbb{F}_q$ be a finite field with characteristic $p$, and $\widehat{\mathbb{F}_q^\times}$ be the group of multiplicative characters of $\mathbb{F}_q^\times$. We follow the convention $\chi(0)=0$, for a non-trivial multiplicative character, whereas for the trivial character $\epsilon_0$, we take $\epsilon_0(0)=1$. 
Then for a multiplicative character $\chi$, define the Gauss sum associated to $\chi$ as  
$$ g_q(\chi)=\sum_{x\in \mathbb{F}_q}\chi(x)\psi_q(x),$$ where $\psi_q(x)=\zeta_p^{Tr_{\mathbb{F}_q/\mathbb{F}_p}(x)}$, for a fixed choice of primitive $p^{th}$ root of unity $\zeta_p$, and $Tr_{\mathbb{F}_q/\mathbb{F}_p}(\cdot)$ being the trace function from $\mathbb{F}_q$ to $\mathbb{F}_p$.\\ 
The Jacobi sum for two characters, $\chi_1,\; \chi_2 \in \widehat{\mathbb{F}_q^\times}$ is defined as follows
$$J_q(\chi_1,\; \chi_2)=\sum_{x\in \mathbb{F}_q} \chi_1(x)\chi_2(1-x).$$

\paragraph{} 
Here are some properties of Gauss and Jacobi sums that can be easily verified from the definition (see \cite[Chapter~8]{Ireland-Rosen}). Let $\chi$ be a non-trivial character and $\epsilon_0$ be the trivial character. We have
\begin{enumerate} 
     \item\label{item 1} $J_q(\epsilon_0,\, \epsilon_0)=q$; \\$J_q(\epsilon_0,\,\chi)=0$; 
    \\ $J_q(\chi,\,{\chi}^{-1})=-\chi(-1)$;\\ $J_q(\chi_1,\,\chi_2)=J_q(\chi_2,\,\chi_1).$  
    \item $g_q(\epsilon_0)=0$; \\
    $g_q(\chi)g_q({\chi}^{-1})=\chi(-1)q$;\; $|g_q(\chi)|=\sqrt{q}$. 
    \item For $\chi_1\not=\chi_2^{-1}$, $J_q(\chi_1,\,\chi_2)=g_q(\chi_1)g_q(\chi_2)/g_q(\chi_1\chi_2).$ 
\end{enumerate} 
\textbf{Remark:} We can define the Jacobi sum for more than $2$ characters in a similar way, with $$ J_q(\chi_1,\chi_2,...,\chi_n)=\sum_{\substack{x_1,x_2,...,x_n\in\mathbb{F}_q\\ x_1+x_2+...+x_n=1}}
\chi_1(x_1)\chi_2(x_2)...\chi_n(x_n).$$ It satisfies analogous properties to the ones listed in \ref{item 1}. (see \cite[Chapter~8]{Ireland-Rosen}).\\ 
  
Before moving to the next section, we record a relation between counting points over finite fields and character sums. 
\begin{prop}\label{a} 
For a finite field $\mathbb{F}_q$, the number of solutions to the equation $x^m=a$, for a positive integer $m$ dividing $q-1$, is $\displaystyle\sum_{\substack{ \chi^m=1\\ \chi\in \widehat{\mathbb{F}_q^\times}} }\chi(a).$ 
\end{prop}

\begin{proof} 
We work for the non-trivial case: $a\neq 1$. Observe that the set $M=\{\chi\in \widehat{\mathbb{F}_{q}^\times} \mid \chi^m=1  \}$ is an order $m$ cyclic subgroup of the group of characters $\widehat{\mathbb{F}_{q}^\times}$. Also, note if the equation $x^m=a$ has one solution $x_0$, then $\{x_0\zeta_m^i\mid\, i\in\{1,2,...,m\} \}$ is the complete solution set for $\zeta_m$ being a primitive $m$-th root of unity in $\mathbb{F}_q$. And by orthogonality relations of the group of characters, the expression $\displaystyle\sum_{\chi^m=1}\chi(a)$ is either  $m(=|M|)$ or $0$ depending on whether $x^m=a$ has a solution or not. This establishes the equality of the two expressions. 
  \end{proof} 

\textbf{Remark:} In the last proposition, the number of solutions for the case when $m\nmid (q-1)$ is the g.c.d. of $m$ and $q-1$. 
  
\subsection[Bookmark-safe title]{Point counting on $x^n+y^n=1$ in $\mathbb{F}_q$}\label{pointcounting} 
For $q$ an odd prime power with  $n\mid (q-1)$, we have  
$$N_q(x^n+y^n=1)=\sum_{a\in \mathbb{F}_q} N_q(x^n=a)N_q(y^n=1-a), $$ 
and by Proposition \ref{a},  
$$N_q(x^n=a)=\sum_{i=0}^{n-1}\chi_n^i(a)$$ for $\chi_n$ a fixed choice of order $n$   primitive character in $\widehat{\mathbb{F}_q^\times}$.\\  
Combining these, we get 
\begin{equation}\label{eq1} 
    N_q(x^n+y^n=1)=\sum_{j=0}^{n-1}\sum_{i=0}^{n-1}J_q(\chi_n^i,\chi_n^j). 
\end{equation} 
Using the properties of the Jacobi sums listed in the section \ref{section Gauss sums}, we rewrite the equation \eqref{eq1} as 
\begin{equation}\label{eq2} 
N_q(x^n+y^n=1)=q+1-\delta_n(-1)n +\sum_{\substack{ i,j\in (\mathbb{Z}/n\mathbb{Z})^* \\ i+j\neq n}}J_q(\chi_n^i,\chi_n^j). 
\end{equation} 
Where $\delta_n(-1)$ is the indicator function to whether $-1$ is a $n$-th power in the finite field and, $(\mathbb{Z}/n\mathbb{Z})^*=\mathbb{Z}/n\mathbb{Z}-\{0\}$ \\
It is worth noting that $\chi_n(\mathbb{F}_q^\times)\subset\{\zeta_n^i\mid 1\leq i\leq n\},$ and hence $J_q(\chi_n^i,\chi_n^j)\in \mathbb{Z}[\zeta_n]$, this turn out to be useful in later discussions.

\paragraph{Remark:} Equation \eqref{eq2} counts the number of points satisfying $x^n+y^n=1$ over the finite field $\mathbb{F}_q$. It can be extended to any finite extension $\mathbb{F}_{q^r}$ of $\mathbb{F}_q$, by extending $\chi_n$ to a primitive order $n$  character in the extension field by composing it with the norm function, namely $\chi'_n=\chi_n\circ N_{\mathbb{F}_{q^r}/\mathbb{F}_q}$, giving 
\begin{equation} 
    \begin{split} 
         N_{q^r}(x^n+y^n=1)&=q^r+1-\delta_n(-1)n - \sum_{\substack{i,j\in (\mathbb{Z}/n\mathbb{Z})^* \\ i+j\neq n}}(-J_{q^r}(\chi_n'^i,\chi_n'^j))\\ 
     &=q^r+1-\delta_n(-1)n - \sum_{\substack{i,j\in (\mathbb{Z}/n\mathbb{Z})^* \\ i+j\neq n}}(-J_{q}(\chi_n^i,\chi_n^j))^r 
    \end{split} 
\end{equation} 
for the number of solutions to the equation over $\mathbb{F}_{q^{r}}$. Here the last equality holds by  the Hasse-Davenport relation $(-g_q(\chi_n))^r=-g_{q^r}(\chi'_n)$ \cite[Chapter~11]{Ireland-Rosen}. 
  
\subsection{Zeta function of a variety} 
  
In this section, we will assume basic knowledge of algebraic geometry. The first two chapters of Hartshorne's book \cite{Hartshorne} provide a good reference if you are unfamiliar with the terminology. In this section and the following sections, whenever we speak of variety, it will always be smooth and projective variety. 
  
\paragraph{}By Weil conjectures \cite{Silverman}, for a projective non-singular variety $V$ defined over $\mathbb{F}_q$, the local Zeta function $Z_q(V;T)$ given by equation \eqref{zeta function of a variety} is a rational function in the variable $T$. The following proposition gives an equivalent statement to the local Zeta function being rational. 
\begin{prop}\label{zeta point counting} \cite[Chapter~11]{Ireland-Rosen}
The Zeta function of variety $V/\mathbb{F}_q$ is a rational function of $T$, and it is of the form $$Z_q(V;T)=\frac{\prod_{i=1}^{n}(1-\alpha_iT)}{\prod_{j=1}^m(1-\beta_jT)}$$ if, and only if, there exist complex numbers $\alpha_i$ and $\beta_j$ such that, for $r
\geq 1$,
$$N_{q^r}(V)=\sum_{j=1}^m\beta_j^r-\sum_{i=1}^n\alpha_i^r$$ for $m$ and $n$ positive integers. 
\end{prop} 

As an example, let us calculate the Zeta function of the projective curve $FC_n: X^n+Y^n+Z^n=0$, which is the projective closure of the curve considered in the section \ref{pointcounting}. Now, $N_{q^r}(FC_n)=N_{q^r}(FC_n)_{Z\neq 0}+(\text {points at infinity})$. Where the number of points at infinity is the number of $(X,1,0)$ such that $X^n=-1$, that is $\delta_n(-1)n$. Therefore, for $q\equiv1\mod n$, equation \eqref{eq2} gives
$$N_{q^r}(FC_n)=q^r+1^r\ - \sum_{\substack{i,j\in (\mathbb{Z}/n\mathbb{Z})^* \\ i+j\neq n}}(-J_q(\chi_n^i,\chi_n^j))^r.$$ 
And hence,
\begin{equation}\label{zeta of n-fermat curve}
     Z_q(FC_n;T)=\frac{\prod_{\substack{i,j\in (\mathbb{Z}/n\mathbb{Z})^* \\ i+j\neq n}}(1+J_q(\chi_n^i,\chi_n^j)T)}{(1-T)(1-qT)}.
\end{equation}

\subsection{Weil conjectures} 
In 1949, Andr\'e Weil made a series of influential conjectures based on his observations of the local Zeta function of the Fermat-type hypersurfaces \cite{Weil-conjuctuers}.

\begin{theorem}[Weil Conjectures]\label{WC} 
Let $V/\mathbb{F}_q$ be a smooth projective variety of dimension $N$. 
\begin{enumerate} 
    \item[(a)] The Zeta function of $V$, $Z_q(V;T)$ is a rational function in variable $T$. 
    \item[(b)] There is an integer $\epsilon,$ called the Euler characteristic of $V,$ such that $$Z_q(V;1/(q^NT))=\pm q^{N\epsilon/2}T^\epsilon Z_q(V;T).$$ 
    \item[(c)] The Zeta function factors as $$ Z_q(V;T)=\frac{P_{q,1}(T)...P_{q,{2N-1}}(T)}{P_{q,0}(T)P_{q,2}(T)...P_{q,{2N}}(T)   },$$ 
    with each $P_{q,i}(T)\in \mathbb{Z}[T]$, with $$P_{q,0}(T)=1-T\hspace{1cm}and\hspace{1cm}P_{q,2N}(T)=1-q^{N}T,$$ 
    and such that for every $1\leq i\leq (2N-1)$, the polynomial $P_{q,i}(T)$ factors over $\mathbb{C}$ as $$P_{q,i}(T)=\prod_{j=1}^{b_i}(1-\alpha_{i,j}T)\hspace{1cm}with\hspace{1cm}|\alpha_{i,j}|=q^{i/2};$$ 
where $| \cdot |$ denotes any complex absolute value, and denoting the degree of $P_{q, i}(T)$ as $b_i$, where $b_i$ represents the $i^{th}$ Betti number of $V$, we note that it is independent of $q$.
\end{enumerate} 
  
\end{theorem} 

\paragraph{}These conjectures have been proved due to the work of Weil, Dwork, M. Artin, Grothendieck, Deligne, and others \cite{Deuring}, \cite{Dwork}, \cite{Katz}. Deligne's proof of Weil conjectures used \'etale cohomologies for varieties to study their geometric properties over finite fields. The polynomials $P_{q, i}(\cdot)$ mentioned in Weil conjectures are characteristic polynomials of the action of Frobenius automorphism (described in the next section) on the $i$-th cohomology groups of the variety. 

\paragraph{}Since we will be discussing only smooth projective curves, the numerator of the local Zeta function of the same would just be $P_{q,1}$, and we will denote its inverse roots as $\alpha_i$'s. Also, note that $\alpha_i$'s are implicitly dependent on the finite field $\mathbb{F}_q$ over which the curve is considered.

\paragraph{} Before moving ahead, let us look at an application for finding the modular form associated with a CM elliptic curve. 
Consider the case of the degree $3$ Fermat curve $FC_3$,  $X^3+Y^3+Z^3=0$. Since it is a genus $1$ projective curve over complex numbers, it is an elliptic curve.\\
\\
\sloppy For a prime power, $q\equiv 2\mod 3$, $N_q(FC_3)=N_q(X^3+Y^3=1)+\{\text{points at infinity}\},$ where $N(X^3+Y^3=1)=\sum_{a\in \mathbb{F}_q}N(X^3=a)N(Y^3=1-a),$ and since $(3,q-1)=1$, we have $N(X^3=a)=1$ for any $a\in \mathbb{F}_q$. Therefore, $$N_q(FC_3)=q+1,$$ where $1$ counts for the point at infinity $(1,-1,0)$.\\ 
\\
For $q\equiv 1\mod 3$, from section \ref{pointcounting} we get
\begin{equation}\label{number of point on FC_3}
    N_q(FC_3)=q+1+J_q(\chi_3,\chi_3)+\overline{J_q(\chi_3,\chi_3)}.
\end{equation} 
Let $\omega_3$ be a cubic root of unity in $\mathbb{F}_q$, then for a point $(x,y,z)$ on $FC_3,$ the set $\{(\omega_3^ix,\omega_3^jy,1)\mid 1\leq i,j\leq 3\}$ also lies on $FC_3$ over $\mathbb{F}_q$. This gives, $N_q(FC_3)\equiv 0\mod 3$, and from equation \ref{number of point on FC_3} we can see that $Re(J_q(\chi_3,\chi_3))\equiv -1\mod 3$.\\
By combining the congruence condition for the Jacobi sum over $\mathbb{F}_q$ with the fact that $J_q(\chi_3,\chi_3)\in \mathcal{O}_{\mathbb{Q}(\zeta_3)} (=\mathbb{Z}[(1+\sqrt{-3})/2])$ and $|J_{q^2}(\chi_3,\chi_3)|=q$, then for $q\equiv 2\mod 3$, we have $q^2\equiv 1\mod 3$ and we can deduce that $J_{q^2}(\chi_3,\chi_3)=q$.
\\
Hence, from Proposition \ref{zeta point counting} and Weil's conjectures, we have

\begin{equation*} 
    N_{q^r}(FC_3)=
    \begin{cases} 
    q^r+1 - (\iota\sqrt{q})^r- (\overline{\iota\sqrt{q}})^r &\text{ for } q\equiv 2\mod 3;\\ 
    q^r+1 -(-J_q(\chi_3,\chi_3))^r-(-\overline{J_q(\chi_3,\chi_3))^r} &\text{ for } q\equiv 1 \mod 3, 
    \end{cases} 
\end{equation*} 
where $\iota=\sqrt{-1}$. We calculate that this elliptic curve $FC_3$ has conductor $27$. The only weight $2$, level $27$ modular form with CM is one with LMFDB label $27.2.a.a$. Therefore, from the Modularity theorem \cite[Chapter~8]{Diamond-Shurman}, we have $$L(FC_3,s)=L(f_{27.2.a.a},s).$$

\subsection[Bookmark-safe title]{Hecke Characters and $L$-functions}\label{Hecke Characters}
Hecke characters for number fields are a natural generalization of the Dirichlet character and have an $L$-function associated with them, which also satisfies a functional equation. They can be defined over any number field, but we will restrict ourselves to CM fields as dealt with in the book \cite{Ireland-Rosen}. In this section, we give a brief introduction to this and consider an example that will be relevant in later sections.\\\\
The number field $K$ is called a CM field if it is a totally complex quadratic extension of a totally real field. Examples include $\mathbb{Q}(\sqrt{-d})/\mathbb{Q}$, where $d$ is a square-free positive integer; another important example would be cyclotomic fields, $\mathbb{Q}(\zeta_n)/\mathbb{Q}(\zeta_n+\overline{\zeta_n})$, where $\zeta_n$ is a primitive $n$-th root of unity.\\
\\
For a CM Galois field $K$, let $\mathcal{O}_K$ be its number ring and $M\subset\mathcal{O}_K$ be an ideal. A Hecke character modulo $M$ is a function ${}_{h}\chi$ from the set of ideals of $\mathcal{O}_K$ to $\mathbb{C}$ satisfying following conditions
\begin{enumerate}
\item ${}_{h}\chi(\mathcal{O}_K)=1$.
\item ${}_{h}\chi(A)\neq 0$ if, and only if, $A$ is relatively prime to $M$.
\item ${}_{h}\chi(AB)={}_{h}\chi(A){}_{h}\chi(B)$.
\item There is an element $\theta=\sum n(\sigma)\sigma\in\mathbb{Z}[Gal(K/\mathbb{Q})]$ such that for $\alpha\in\mathcal{O}_K$, $\alpha\equiv 1\mod M$, then ${}_{h}\chi((\alpha))=\alpha^{\theta}$.
\item There is an integer $m>0$ called weight, such that for any $\sigma\in Gal(K/\mathbb{Q})$, $n(\sigma)+n(j\sigma)=m.$ Where $j$ denotes the complex conjugation automorphism restricted to $K$.
\end{enumerate}
The largest ideal $M$ such that Hecke character $\chi$ is a character modulo $M$ is called the conductor of the character.
\begin{prop}\cite[Chapter~18]{Ireland-Rosen} \label{weight of Hecke character}
    Let ${}_{h}\chi$  be a Hecke character of weight $m$ and modulus $M$. Then if $(A,M)=1$, $|{}_{h}\chi(A)|=N(A)^{m/2}.$
\end{prop}
The $L$-function attached to a Hecke character $\chi_M$ of $K$ is defined as
$$L({}_{h}\chi,s):=\prod_{\mathfrak{p}}(1-{}_{h}\chi(\mathfrak{p})
N(\mathfrak{p})^{-s})^{-1}=
\sum_A\frac{{}_{h}\chi(A)}{N(A)^s}.$$ Where the product is over all the prime ideals in $\mathcal{O}_K$ and the sum is over all the ideals in $\mathcal{O}_K.$

\begin{theorem}\cite{SilvermanAdvancedTI}
    Let $\chi_M$ be a Hecke character of weight $m$, then

\begin{itemize}
    \item $L({}_{h}\chi,s)$ converges absolutely for $Re(s)>1+m/2$.
     \item $L({}_{h}\chi,s)$ has a meromorphic continuation to the whole complex plane, being analytic except for a pole of order 1 at $s=1$ when the character is trivial.
     \item If ${}_{h}\chi$ is a primitive character, $L({}_{h}\chi,s)$ satisfies a function equation relating $L({}_{h}\chi,s)$ to $L(\overline{{}_{h}\chi},1+m-s)$.
\end{itemize}
\end{theorem}

Due to Deuring's work, we know that  $L$-functions of a CM elliptic curve are equal to the $L$-function of a Hecke character over a number field. To demonstrate this, let $K=\mathbb{Q}(\sqrt{-7})$, and define the map from the prime ideals of $\mathcal{O}_K$ to $\mathbb{C}$  as:
\begin{itemize}
\item ${}_{h}\chi((\sqrt{-7}))=0$.
    \item ${}_{h}\chi(P)=-p$, for $P=(p)$, where $p$ is  an inert prime.
    \item For primes $p$ splitting in $\mathcal{O}_K$, that is $p\equiv 1,2$ or $4\mod 7$, let $p=\left(\frac{a+b\sqrt{-7}}{2}\right)\left(\frac{a-b\sqrt{-7}}{2}\right)$, then for prime ideal $P=(\frac{a+b\sqrt{-7}}{2})$ let,
    $${}_{h}\chi(P)=\frac{a+b\sqrt{-7}}{2}, \text{ with }\;\, {}_{h}\chi(P)p\equiv 1\mod\sqrt{-7}.$$

    \end{itemize}

This map is a Hecke character in $\mathbb{Q}(\sqrt{-7})$, with conductor $(\sqrt{-7})$.\\~\\
Later, we realize that the $L$-function of this Hecke character is the $L$-function of the newform with integer Fourier coefficients, which appears as a factor of the $L$-function of the Klein Quartic curve.\\~\\

\subsection{Galois representations}\label{Grep} 
  
We follow \cite{Diamond-Shurman} for this section. Some background from class field theory is helpful, a good reference would be \cite{Neukirch} or \cite{Lang}. Let $\mathbb{K}$ be a Galois number field, with $\mathcal{O}_{\mathbb{K}}$ its ring of integers. Then for a prime $p\in \mathbb{Z}$, there exists positive integers $e$, $f$ and $g$ describing $p\mathcal{O}_{\mathbb{K}}$ in terms of product of distinct maximal ideals $\mathfrak{p}_i$'s in $\mathcal{O}_{\mathbb{K}}$ as 
$$ p\mathcal{O}_\mathbb{K}=(\mathfrak{p}_1\mathfrak{p}_2...\mathfrak{p}_g)^e,$$ and  
$$\mathcal{O}_\mathbb{K}/\mathfrak{p}_i\cong\mathbb{F}_{p^f} \; \text{for all} \;i,$$ with product $efg=[\mathbb{K}:\mathbb{Q}].$ Define the decomposition group for each maximal ideal $\mathfrak{p}$ in $\mathcal{O}_\mathbb{K}$ lying over $p\in \mathbb{Z}$ as  
$$D_\mathfrak{p}:=\{ \sigma\in Gal(\mathbb{K}/\mathbb{Q}) \,\mid\, \mathfrak{p}^{\sigma}=\mathfrak{p}  \}.$$ 
Taking the kernel of the natural action of the decomposition group on the residue field $\mathbf{f}_\mathfrak{p}=\mathcal{O}_{\mathbb{K}}/\mathfrak{p}$, gives the inertia group of $\mathfrak{p}$ 
$$I_\mathfrak{p}:=\{  \sigma\in D_{\mathfrak{p}}\  \,\mid\, x^{\sigma}\equiv x \mod \mathfrak{p}, \,\forall\, x \in \mathcal{O}_{\mathbb{K}}    \}.$$ 
It turns out that the quotient $D_\mathfrak{p}/I_\mathfrak{p}$ is isomorphic to $Gal(\mathbf{f}_{\mathfrak{p}}/\mathbb{F}_p )= \langle \sigma_{p} \rangle$, where $\sigma_p$ is the $p$-th power map. 
A representative of the generator of the quotient group $D_\mathfrak{p}/I_\mathfrak{p}$ is called a Frobenius element for the maximal ideal $\mathfrak{p}$, denoted as $\text{Frob}_{\mathfrak{p}}$. \\

Now let $G_{\mathbb{Q}}$ be the absolute Galois group of $\mathbb{Q}$. The decomposition and inertia group of a maximal ideal $\mathfrak{p}$ in $\overline{\mathbb{Z}}$ over a prime $p$ is defined similarly as in the case of a number field. With 
$$D_{\mathfrak{p}}=\{ \sigma\in G_\mathbb{Q}  \,\mid\, \mathfrak{p}^{\sigma}=\mathfrak{p}  \}$$ and, 
$$I_{\mathfrak{p}}=\{ \sigma\in D_{\mathfrak{p}} \,\mid\, x^{\sigma}\equiv x \mod \mathfrak{p}, \,\forall\, x \in \overline{\mathbb{Z}} \}.$$ 
Where $I_{\mathfrak{p}}$ is again the kernel of the action of $D_{\mathfrak{p}}$ on the residue field $\overline{\mathbb{Z}}/\mathfrak{p}$. The homomorphism obtained by this action $$D_{\mathfrak{p}}\to G_{\mathbb{F}_{p}},$$ 
is surjective, and any representative of the preimage of Frobenius automorphism $\sigma_p\in G_{\mathbb{F}_p}$, denoted as $\text{Frob}_{\mathfrak{p}}$, is called an absolute Frobenius element.  
Absolute Frobenius element(s) when restricted to a number field $\mathbb{K}$ correspond(s) to a Frobenius element of the number field, i.e.,
$$\text{Frob}_{\mathfrak{p}}|_{\mathbb{K}}=\text{Frob}_{\mathfrak{p}_{\mathbb{K}}},\,\,\,\, \text{with}\,\, \mathfrak{p}_\mathbb{K}=\mathfrak{p}\cap \mathbb{K}.$$ 
All maximal ideals over a prime $p$ are conjugate to each other, and Frobenius of a conjugate ideal is a conjugate of Frobenius element, i.e.

$$\text{Frob}_{\mathfrak{p}^{\sigma}}=\sigma^{-1}\text{Frob}_{\mathfrak{p}}\sigma.$$ 
\begin{theorem}\label{density} \cite{Diamond-Shurman}
The set of $\text{Frob}_{\mathfrak{p}}$ for all but finitely many maximal ideals $\mathfrak{p}$ of $\overline{\mathbb{Z}}$, form a dense subset of the absolute Galois group $G_{\mathbb{Q}}$. 
  
\end{theorem} 
  
Let $\rho$ be a finite-dimensional, continuous, $\ell$-adic Galois representation, then we call this representation \textit{unramified} at the ideal $\mathfrak{p}$, if $I_{\mathfrak{p}}\subset Ker(\rho)$. For $\rho$ unramified at $\mathfrak{p}$, value of $\rho({\text{Frob}_{\mathfrak{p}}})$ is well defined. From the Theorem \ref{density} we see that  $\rho$ is determined by the values of $\rho(\text{Frob}_{\mathfrak{p}})$ at the maximal ideals $\mathfrak{p}$, where $\rho$ is unramified.

\section{Klein Quartic curve} 
Klein Quartic curve, defined by the following equation
\begin{equation}\label{KQ curve} 
    x^3y+y^3z+z^3x=0, 
\end{equation} 
is a compact Riemann surface of genus three. It has the highest possible order of the automorphism group (orientation preserving) for its genus, given by Hurwitz's automorphism theorem \cite{Hurwitz}, with an automorphism group of order $84(g-1)=168$.
To mention some of the automorphisms, notice that any cyclic interchange of coordinates of a point still lands on the curve; therefore, the cyclic group $C_3$ is contained in the automorphism group of the Klein Quartic curve. 
Further, notice that $\tau_{(2,1,4)}((x,\, y,\, z))=(\zeta_7^{2k} x,\,\zeta_7^{k} y, \zeta_7^{4k} z)$ is also an automorphism for $1\leq k\leq 7$. This automorphism turns out to play a key role in determining the Zeta function of the curve. The automorphism group of the Klein Quartic curve is isomorphic to $\text{PSL}_2(\mathbb{F}_7)$, which is the second smallest non-abelian simple group. For a more detailed discussion of the automorphism group, the reader can refer to the essay on the Klein Quartic curve by Noam D. Elkies \cite{Elkies}.

\subsection[Bookmark-safe title]{Number of points modulo a prime $p$ on Klein Quartic curve over finite fields of characteristic $p$.} \label{section point counting modulo p}

The Hasse-Weil theorem for smooth projective curve $C$ of genus $g$, defined over a finite field $\mathbb{F}_q$, gives the following bound for the number of points on the curve over $\mathbb{F}_q$, $$|N_q(C)-(q+1)|\leq 2g\sqrt{q}.$$
Therefore, for large $q$,  counting the number of points up to a multiple of $p$ on the curve over the finite field $\mathbb{F}_q$ of characteristic $p$ can give us information about the exact number of points over $\mathbb{F}_q$.

\begin{lemma}Let $f(x,y,z)=x^3y+y^3z+z^3x$ defined over a finite field $\mathbb{F}_q$ of characteristic $p\neq 7$. Then
\begin{equation}\label{point counting modulo q} 
N_q(f)\equiv 
\begin{cases} 
1-3\displaystyle \binom{q-1}{\frac{q-1}{7},\frac{2(q-1)}{7},\frac{(4q-1)}{7}} &\mod p \hspace{1cm} \text{ if } 7\mid (q-1);\\ 
1 &\mod p \hspace{1cm}\text{ otherwise}. 
\end{cases} 
\end{equation} 
\end{lemma}
\begin{proof}
A point $(x,y,z)$ over $\mathbb{F}_q$ is not a solution of $f(x,y,z)=0$ if, and only if, $f(x,y,z)^{q-1}= 1$ in $\mathbb{F}_q$. The number of solutions of $f$ in degree-2 projective space over finite field $\mathbb{F}_q$ is given by 
\begin{equation} 
    N_q(f)= \left[q^2-\sum_{\substack {(x,y)\in \mathbb{F}_q^2}}f(x,y,1)^{q-1}\right]+2.
\end{equation} 
The expression in the square brackets is the number of points with $z=1$; and there are two points at infinity on the curve, namely $(1,0,0)$ and $(0,1,0)$.\\ 
Rewriting the above expression, by expanding $f(x,y,z)^{q-1}$, 
\begin{equation} 
    N_q(f)\equiv -\sum_ {(x,y)\in \mathbb{F}_q\times \mathbb{F}_q}\sum_{i,j}a_{ij}x^iy^j +2  \;\; \mod p. 
\end{equation} 
The only terms that remain in the sum will be the ones with $(q-1)\mid i$ and $(q-1)\mid j$. Otherwise, it adds up to zero using the fact that for $(q-1) \nmid i$, the sum $\sum_{x\in\mathbb{F}_q}x^i=0$. \\
\\
The possible terms in the sum of the mentioned form are $\{x^{3q-3}y^{q-1},\,y^{3q-3},\,x^{q-1}\}$ and $ \{x^{q-1}y^{2q-2},\,x^{2q-2}y^{q-1},\,x^{q-1}y^{q-1}\}$. In the first set, the term $x^{3q-3}y^{q-1}$ adds up to $(q-1)^2\equiv 1\mod p$, while others add up to $q(q-1)\equiv0\mod p$. Whereas the terms in the second set occur if, and only if, $7$ divides $q-1$, giving an overall contribution of $$3\binom{q-1}{\frac{q-1}{7},\frac{2(q-1)}{7},\frac{4(q-1)}{7}}.$$ 
Combining the contribution from both sets gives the expression for $N_q(f)$ \eqref{point counting modulo q}.    
\end{proof}

A quick observation that we can make from this lemma is that for $\alpha_i$ being the inverse roots of polynomial $P_{q,1}$, and $q\not\equiv 1\mod 7$, we have $\sum_{i=1}^6\alpha_i\equiv0\mod p$. In fact, in section \ref{Zeta function of the Klein Quartic curve} we see that we can give a morphism from $FC_7$ to $KQ$ over any field, in particular, it is a bijection if the field does not contain a $7$-th root of unity, which is the case for $\mathbb{F}_q$ with $7\nmid (q-1)$. We have $N_q(FC_7)=q+1$, and hence $N_q(KQ)=q+1$. This agrees with the lemma, furthermore, this gives $\sum_{i=1}^6\alpha_i=0$.

\subsection{Zeta function of the Klein Quartic curve}\label{Zeta function of the Klein Quartic curve}
The Klein Quartic curve is a smooth projective curve of genus three over the complex numbers; therefore, for $7\nmid q$, by Weil's conjectures, we see that its Zeta function is of the form 
\begin{equation}\label{ZetaKQ} 
    Z_q(KQ;T)=\frac{\prod_{i=1}^6(1-\alpha_iT)}{(1-T)(1-qT)}, 
\end{equation} 
with $|\alpha_i|=q^{1/2}$, and $q$ a prime power. 
Also, observe the map $$\phi:\, X^7+Y^7+Z^7=0\to x^3y+y^3z+z^3x=0,$$ given by $$\phi(X,\, Y,\, Z)=(X^3Z,\, Y^3X, \, Z^3Y),$$ is a non-constant morphism of curves, hence it is surjective. Therefore, every point in the Klein Quartic is of the form $(a^3c,\,b^3a,\,c^3b)$ for some $(a,\,b,\,c)$ in the degree $7$ Fermat curve. We observe that for any fixed choice of $7$-th root of unity $\zeta_7\in \mathbb{C}$, there are at least $7$ preimages for the point $(a^3c,\,b^3a,\,c^3b)$, namely the set,
$\{(\zeta_7^ia,\,\zeta_7^{2i}b,\,\zeta_7^{4i}c)\mid 1\leq i\leq 7\}$.
Therefore, the degree of the map has to be at least $7$; and from the Riemann-Hurwitz theorem \cite{Silverman}, we see that the maximum possible degree is $7$, for which it must be unramified. In conclusion, Fermat curve $FC_7$ is a degree $7$ unramified cover of the Klein Quartic curve. 
  
\paragraph{ }Putting things together, we now know that Klein Quartic's $L$-function must be a factor of the $L$-function of the degree $7$ Fermat curve.  In other words, the $6$ factors in the numerator of the local Zeta function of the Klein Quartic curve come from the $30$ factors in the numerator of the local Zeta function of the degree $7$ Fermat  Curve. Where the numerator for degree $7$ Fermat curve, given by \eqref{zeta of n-fermat curve}, is $P_{q,1}(T)=\prod_{\substack{i,j=1\\ i+j\neq7}}^6(1-J_q(\chi_7^i,\chi_7^j)T)$. All that remains to see is which $6$ out of these $30$ factors are present in the local Zeta function of the Klein Quartic curve.  
  
\paragraph{} For $q\equiv1\mod 7$, we have $J_q(\chi_7^i,\chi_7^j)\in \mathbb{Z}[\zeta_7]$, and the Galois group $G_{\mathbb{Q}}$ has an action on the set of Jacobi sums $\{J_q(\chi_7^i,\, \chi_7^j) \mid 1\leq i,\,j\leq 6;\, i+j\neq 7\}$, via $\sigma_k(\zeta_7)=\zeta_7^k$. This action divides the set into $5$ orbits of $6$ elements each. One of these orbits corresponds to the local Zeta function of the Klein Quartic curve.\\ 
We can give the action of $G_{\mathbb{Q}}$  on the points of the Fermat curve as $$\tau_{(i,\,j,\,k)}( X,\,Y,\, Z)= (\zeta_7^i X,\,\zeta_7^{j} Y, \zeta_7^{k}Z), $$ with $i,j,k$ varying from $1$ to $6$. 
Whereas, for $\tau_{(i,j,k)}$ to be an action on the Klein Quartic curve we want $(\zeta_7^i x,\,\zeta_7^{j} y, \zeta_7^{k}z)$ to be a point on the curve \eqref{KQ curve}. As stated at the beginning of this section only $\tau_{(2k,k,4k)}$, for $1\leq k\leq 6$, induces automorphisms on the curve. 

\paragraph{}Before deciding on the orbit we record that for $7\mid (i+j+k)$ and $7\nmid i,j,k$  we have $J_q(\chi_7^i,\chi_7^j)=J_q(\chi_7^i,\chi_7^k)=J_q(\chi_7^j,\chi_7^k)$. To see this observe
\begin{align} 
    J_q(\chi_7^i,\chi_7^j)&=\frac{g_q(\chi_7^i)g_q(\chi_7^j)}{g_q(\chi_7^{i+j})}\\ 
    &=\frac{g_q(\chi_7^i)g_q(\chi_7^j)g_q(\chi_7^k)}{g_q(\chi_7^{i+j})g_q(\chi_7^k)}\\ 
    &=\frac{g_q(\chi_7^i)g_q(\chi_7^j)g_q(\chi_7^k)}{q}
\end{align} 
We get the same expression for all three Jacobi sums, hence, they are equal, and so are their complex conjugates. In particular, for $(i,j,k)=(1,2,4)$ this implies $\sigma_2(J_q(\chi_7,\chi_7^2))=J_q(\chi_7,\chi_7^2)$.
  
\paragraph{} From the discussions above we can conclude that the orbit of the set of Jacobi sums for the $Z_q(KQ;\, T)$ is $\{J_q(\chi_7^{2k}, \,\chi_7^{k})$, for $1\leq k\leq6\}$. Where $J_q(\chi_7^{2k},\chi_7^{k})$ is equal to $g_q(\chi_7^k)g_q(\chi_7^{2k})g_q(\chi_7^{4k})/q$, for which $\chi_7^k,\chi_7^{2k}$ and $\chi_7^{4k}$ play symmetric roles. Therefore, the local Zeta function for the Klein Quartic curve for $q\equiv 1\mod 7$ is 
\begin{equation}\label{Zeta KC} 
    Z_q(KQ;\,T)=\frac{\prod_{i=1}^6 (1+J_q(\chi_7^i,\,\chi_7^{2i})T)}{(1-T)(1-qT)}. 
\end{equation} 
This gives that for $q\equiv 1\mod 7$ the number of points on the Klein Quartic curve over $\mathbb{F}_q$  is  
\begin{equation}\label{points on KQ} 
     N_q(KQ)=q+1-\sum_{i=1}^6J_q(\chi_7^i,\chi_7^{2i}). 
\end{equation} 
Before moving ahead, we make some observations regarding the value of the Jacobi sum $J_q(\chi_7,\chi_7^{2})$.

\begin{lemma}\label{Jacobi sum congruence}
Let $\mathbb{F}_q$ be a finite field, such that $q\equiv1\mod 7$. Then for a fixed choice of primitive character $\chi_7\in\widehat{\mathbb{F}_q^\times}$,the Jacobi sum $J_q(\chi_7,\chi_7^2)=u+\sqrt{-7} v$ with $u,v\in \mathbb{Z}$, and $u\equiv -1\mod 7.$
\end{lemma}

\begin{proof}
Since $J_q(\chi_7,\chi_7^2)$ is a polynomial in $\zeta_7$ over $\mathbb{Z}$, it is an element in the ring of integers of $\mathbb{Q}(\zeta_7)$, i.e. $J_q(\chi_7,\chi_7^2)\in\mathbb{Z}[\zeta_7]$. Combining this with the fact that it is invariant under $\sigma_2$ implies that it lies in a degree $2$ extension of $\mathbb{Q}$. The only quadratic sub-field of  $\mathbb{Q}(\zeta_7)$ is $\mathbb{Q}(\sqrt{-7})$, hence we conclude that $J_q(\chi_7,\chi_7^2)\in\mathcal{O}_{\mathbb{Q}(\sqrt{-7})}$. Now, let $J_q(\chi_7,\chi_7^2)=U/2+\sqrt{-7}V/2$ with $U,V\in\mathbb{Z}$. Then we have $4q=U^2+7V^2$, with $q\equiv1\mod 14$; this holds if, and only if, $2 \mid U$ and $2\mid V$. Therefore, $J_q(\chi_7,\chi_7^2)\in \mathbb{Z}[\sqrt{-7}].$ \\
Now for the congruence part, let $J_q(\chi_7,\chi_7^2)=u+\sqrt{-7}v$, then the fact that $q=u^2+7v^2$ gives, $u\equiv\pm 1\mod7.$
Further notice that for $q\equiv1\mod 7$ and $(x,y)\in \mathbb{F}_q^{\times}$, if a point $(x,y,1)$ is on the Klein Quartic curve over $\mathbb{F}_q$, then $\{(\zeta_7^{2k}x,\zeta_7^{k}y,z)\mid 1\leq k \leq7 \}$ lie on the curve. Therefore, points on the curve with $x,y,z\neq 0$ exists as a set of $7$ elements, other than this, there are $3$ points namely $(0,0,1),(1,0,0),$ and $(0,1,0)$. This gives
\begin{equation}\label{KQ points q=1 congruence}
    N_q(FC_7)\equiv 3\mod 7,
\end{equation}
 and hence $u\equiv -1\mod 7.$
\end{proof}

\paragraph{}Let us verify that for a prime $p\equiv1\mod 7$, the expression for $N_p(KQ)$ \eqref{points on KQ}, satisfies the congruence in the equation \eqref{point counting modulo q}. Equivalently, we want $-(J_p(\chi_7,\chi_7^2)+\overline{J_p(\chi_7+\chi_7^2)})\equiv \displaystyle \binom{p-1}{\frac{p-1}{7},\frac{2(p-1)}{7},\frac{4(p-1)}{7}} \mod  p$.
From lemma \ref{Jacobi sum congruence}, we have $J_p(\chi_7,\chi_7^2)=u+v\sqrt{-7}$, for $x,y\in \mathbb{Z}$, and hence $-(J_p(\chi_7,\chi_7^2)+\overline{J_p(\chi_7,\chi_7^2)})=-2u$. Whereas, from  \cite{Hudson-Williams}, we have $\displaystyle\binom{\frac{3(p-1)}{7}}{\frac{p-1}{7}}\equiv -2u \mod p$. Now observe
\begin{align} 
    \binom{p-1}{\frac{p-1}{7},\frac{2(p-1)}{7},\frac{4(p-1)}{7}}&\equiv \frac{(p-1)!}{\frac{p-1}{7}!\frac{2(p-1)}{7}!\frac{4(p-1)}{7}!}&\mod p\\ 
    &\equiv\binom{\frac{3(p-1)}{7}}{\frac{p-1}{7}}\frac{(p-1)!}{ \frac{3(p-1)}{7}! \frac{4(p-1)}{7}!}&\mod p\\ 
    &\equiv -2u&\mod p.
\end{align} 
Hence, the congruence relation \eqref{point counting modulo q} holds.\\

We looked at the Klein Quartic curve as a quotient of degree $7$ Fermat curve to find the local Zeta function, One can also perform the point counting over finite fields for any birationally equivalent form to obtain the Zeta function. In fact, in \cite{Elkies}, Elkies shows that the affine Klein Quartic curve is birationally equivalent over $\mathbb{Q}$ to the curve $y^7=x^2(x+1)$. Therefore, their projective closure will have the same number of points over any finite field, and hence, they have the same local Zeta function. For $q\equiv1\mod 7$, 
\begin{align*} 
N_{\mathbb{F}_q}(KQ)=1+\sum_{i=1}^7\sum_{x\in\mathbb{F}_q}\chi_7^{2i}(x)\chi_7^i(1-x)=1+q+\sum_{i=1}^6J_q(\chi_7^i,\chi_7^{2i}) 
\end{align*} 
This again confirms our calculation for the local Zeta function at primes powers, $q\equiv1\mod 7$.

\section{Modular forms in the background} 
  
\paragraph{} In \cite{Elkies}, Elkies gives the decomposition of the Jacobian of the Klein Quartic curve in terms of elliptic curves, Alternatively, we look at the Galois representation to find if there are any modular forms in the background.
Let $K=\mathbb{Q}(\zeta_7)$ be the cyclotomic extension of degree $6$ over $\mathbb{Q}$. The idea is to give a one-dimensional representation of the Galois group $G_{\mathbb{Q}(\zeta_7)}$; extend it to a representation of the absolute Galois group $G_{\mathbb{Q}}$ using the method of induced representations \cite{Curtis-Reiner}, and then look for the automorphic forms associated to the Galois representation of the group $G_{\mathbb{Q}}$. 
  
\paragraph{} 
Due to Weil's theory \cite{Weil-Jacobi-sums}, Jacobi sums can be considered as Hecke characters, also described in \cite{Long-Tu}. 
More precisely, for a fixed positive integer $M$, let $\mathfrak{p}$ be an unramified maximal ideal in $\mathbb{Z}[\zeta_M]$, and $q(\mathfrak{p})=|\mathbb{Z}[\zeta_M]/\mathfrak{p}|$. Then for fixed choice of rational numbers $a$ and $b$,  such that $a(q(\mathfrak{p})-1)$ and $b(q(\mathfrak{p})-1)$ are integers, the function from the set of all fractional ideals of $\mathbb{Z}[\zeta_M]$, coprime to $M$, to $\mathbb{C}^\times$, defined by the following formula, is a Hecke character of modulus $M^2$ (see \cite{Weil-Jacobi-sums}) 
$$\mathcal{J}_{a,b}(\mathfrak{p})=-J_{q(\mathfrak{p})}(\chi_{{\mathfrak{p}},M}^{a},\chi_{\mathfrak{p},M}^{b}).$$ 
Here $\chi_{\mathfrak{p},M}$ is a primitive character of order $M$ in the residue field $\mathbb{Z}[\zeta_M]/\mathfrak{p}$, defined as  
  
\begin{equation}\label{n-th residue symbol} 
     \chi_{\mathfrak{p},M}(x)\equiv x^{(q(\mathfrak{p})-1)/M} \mod \mathfrak{p},\;\;\forall\, x\in \mathbb{Z}[\zeta_M] .     
\end{equation}

In terms of Galois representation, this means for any fixed prime $\ell$, there is a one-dimensional Galois representation $G_{\mathbb{Q}(\zeta_M)}\to \overline{\mathbb{Q}_\ell}$, unramified for primes $\mathfrak{p}$ not dividing $M$ is given as $$\rho_{a,b}(\text{Frob}_{\mathfrak{p}})=-J_{q(\mathfrak{p})}(\chi_{\mathfrak{p}, M}^{a},\chi_{\mathfrak{p}, M}^{b}),$$ 
where $\text{Frob}_{\mathfrak{p}}$ is the Frobenius element associated with the ideal $\mathfrak{p}$ (see section \ref{Grep}).

\paragraph{} 
In the case of the Klein Quartic curve, the Hecke character associated with the Jacobi sum $J_{q(\mathfrak{p})}(\chi_{\mathfrak{p},7},\chi_{\mathfrak{p},7}^2)$ has modulus $(49)$ (not necessarily primitive).  Weil in his paper \cite{Weil-Jacobi-sums} also remarks that the conductor for Jacobi sum character in $\mathbb{Q}(\zeta_M)$ is  either $(1-\zeta_M)$ or $(1-\zeta_M)^2$, for an odd prime $M$.  A quick check based on proposition \ref{weight of Hecke character} reveals that the weight of our Hecke character is $1$. Using this and the fact that the Jacobi sums under consideration are invariant under $\sigma_2$, implies that for a maximal ideal $\mathfrak{p}=(\alpha_\mathfrak{p})$, ${}_h\chi(\mathfrak{p})=\phi(\alpha_\mathfrak{p})\alpha_\mathfrak{p}^\theta$. Where $\theta$ is either  $(I+\sigma_2+\sigma_4)$ or $(\sigma_3+\sigma_5+\sigma_6)$, with the value of $\theta$ invariant of $\mathfrak{p}$. And from \ref{Jacobi sum congruence} we get $\phi(\alpha_\mathfrak{p})=1$ for $\alpha_\mathfrak{p}^\theta\equiv1\mod(1-\zeta_7)$, and $-1$ otherwise. This can be extended to any ideal $(\alpha)$ coprime to $(1-\zeta_7)$, which gives ${}_h\chi((\alpha))=\alpha^\theta$ for $\alpha\equiv1\mod(1-\zeta_7)$. Hence, we conclude that the conductor of our character is $(1-\zeta_7).\textit{}$

\paragraph{}The one-dimensional Galois representation $\rho: G_{\mathbb{Q}(\zeta_7)}\to \overline{\mathbb{Q}_\ell}$ on the Frobenius elements of maximal ideals $\mathfrak{p}$ coprime to $7$ is given as 
  
\begin{equation}\label{1-dim rep} 
    \rho(\text{Frob}_{\mathfrak{p}})=-J_{q(\mathfrak{p})}(\chi_{\mathfrak{p},7},\chi_{\mathfrak{p},7}^2). 
\end{equation} 
We know from Modularity theorem \cite{Diamond-Shurman}, that an irreducible $2$-dimensional, odd (i.e. for complex conjugate automorphism $\sigma^c$, $det(\rho(\sigma^c)) =-1$), continuous  $G_{\mathbb{Q}}$-representations that are ramified at only finitely many places arises from modular forms.  
To find out the associated automorphic forms, we induce the one-dimensional representation $\rho$ to the absolute Galois group $G_{\mathbb{Q}}$. We do this in two steps: first, we induce it to a  
$2$-dimensional representation of the group $G_{\mathbb{Q}(\zeta_7+\overline{\zeta_7})}$, say $\psi$, and finally induce $\psi$ to  a $6$-dimensional representation of $G_{\mathbb{Q}}$. 
\[ 
\newcommand{\ext}[1]{
  \hphantom{\scriptstyle#1}\bigg|{\scriptstyle#1}%
} 
\begin{array}{@{}c@{}} 
\overline{\mathbb{Q}} \\ 
\ext{} \\ 
\mathbb{Q}(\zeta_7) \\ 
\ext{2} \\ 
\mathbb{Q}(\zeta_7+\overline{\zeta_7}) \\ 
\ext{3} \\ 
\mathbb{Q} 
\end{array} 
\] 
Whether the induced representation  
$$\psi:=\text{Ind}_{G_{\mathbb{Q}(\zeta_7)}}^{G_{\mathbb{Q}(\zeta_7+\overline{\zeta_7})}}\rho$$ 
is irreducible or not, is enough to checking that $\rho$ and its complex conjugate, denote by $\rho^c$, are not equivalent as representations of $G_{\mathbb{Q}(\zeta_7+\overline{\zeta_7})}$ \cite{Curtis-Reiner} . Let $\sigma^c$ denote the complex conjugation, which is the non-trivial representative of the coset $G_{\mathbb{Q}(\zeta_7+\overline{\zeta_7})}/G_{\mathbb{Q}(\zeta_7)} $. Then $\rho^c(x)=\rho(\sigma^cx\sigma^c)$. Since $\rho$ and $\rho^c$ are $1$-dimensional representations, to check that they are inequivalent is enough to check that they disagree on any one element. We have $\sigma^c\text{Frob}_{\mathfrak{p}}\sigma^c=\text{Frob}_{\mathfrak{p}^c}$ (see section \ref{Grep}); therefore, $$\rho^c(\text{Frob}_{\mathfrak{p}})=\rho(\sigma^c\text{Frob}_{\mathfrak{p}}\sigma^c)=\rho(\text{Frob}_{\mathfrak{p}^c})=-\overline {J_{q(\mathfrak{p})}(\chi_{\mathfrak{p},7},\chi_{\mathfrak{p},7}^2) }.$$ 
By direct calculation, we verify that the Jacobi sum for $p=29$ is not a real number. This proves that $\rho^c$ and $\rho$ are inequivalent, and hence $\psi$ is irreducible. 
  
\paragraph{} Next we induce $\psi$ to the group $G_{\mathbb{Q}}$. We have $G_\mathbb{Q}/G_{\mathbb{Q}(\zeta_7+\overline{\zeta_7})}=\{1,\sigma_2,\sigma_2^2\}$. Let  
$$\theta:=\text{Ind}_{G_{\mathbb{Q}(\zeta_7+\overline{\zeta_7})}}^{G_{\mathbb{Q}}}\psi$$ 
be the induced $6$-dimensional representation of the absolute Galois group. We claim that this is a direct sum of three irreducible $2$-dimensional representations. To see this, it is sufficient to show that the characters of the conjugate representations of $\psi$ agree on the Frobenius elements. The conjugate representations to $\psi$ are $\psi^{\sigma_2}$ and $\psi^{\sigma_2^2}$, and their action on an element is given by the action of $\psi$ on the conjugate of that element by $\sigma_2$ and $\sigma_2^2$, respectively. This gives 
\begin{align*} 
    \chi_{\psi}(\text{Frob}_{\mathfrak{p}})&=-(J_{q(\mathfrak{p})}(\chi_{\mathfrak{p},7},\chi_{\mathfrak{p},7}^2)+ \overline{J_{q(\mathfrak{p})}(\chi_{\mathfrak{p},7},\chi_{\mathfrak{p},7}^2}));\\ 
    \chi_{\psi^{\sigma_2}}(\text{Frob}_{\mathfrak{p}})&=-(J_{q(\mathfrak{p})}(\chi_{\mathfrak{p},7}^2,\chi_{\mathfrak{p},7}^4)+ \overline{J_{q(\mathfrak{p})}(\chi_{\mathfrak{p},7}^2,\chi_{\mathfrak{p},7}^4}));\\ 
    \chi_{\psi^{\sigma_2^2}}(\text{Frob}_{\mathfrak{p}})&=-(J_{q(\mathfrak{p})}(\chi_{\mathfrak{p},7}^4,\chi_{\mathfrak{p},7})+ \overline{J_{q(\mathfrak{p})}(\chi_{\mathfrak{p},7}^4,\chi_{\mathfrak{p},7}})). 
\end{align*} 
From the discussion in section \ref{Zeta function of the Klein Quartic curve}, all three  sums in the above equations are equal to $J_{q(\mathfrak{p})}(\chi_{\mathfrak{p},7},\chi_{\mathfrak{p},7}^2,\chi_{\mathfrak{p},7}^4)+ \overline{J_{q(\mathfrak{p})}(\chi_{\mathfrak{p},7},\chi_{\mathfrak{p},7}^2,\chi_{\mathfrak{p},7}^4)}$. From this we know that $\psi$ is extendable to $G_{\mathbb{Q}}$, and let $\Tilde{\psi}$ be the trivial extension (i.e. $\Tilde{\psi}(\sigma_2^ih)=\psi(h)$ for $i\in\{1,2,3\}$ and $h\in G_{\mathbb{Q}_{\zeta_7+\overline{\zeta_7}}}$). The inverse of an open set $U$ under the map $\Tilde{\psi}$ is given as,  $\Tilde{\psi}^{-1}(U)=\cup_{i=0}^2(\sigma_2^{i}\psi^{-1}(U))$, which is a union of open sets, and hence the extension is continuous.  This gives that the induced representation $\theta$ is a direct sum of three irreducible $2$-dimensional representations, 
$$\theta=(\Tilde{\psi}\otimes \epsilon)\oplus (\Tilde\psi\otimes \eta)\oplus (\Tilde\psi\otimes\eta^2), $$ 
with $\Tilde\psi$ being the trivial extension of $\psi$ to $G_\mathbb{Q}$, and $\eta$ is an order $3$ character of the group $G_{\mathbb{Q}} $  with kernel $ {G_{\mathbb{Q}(\zeta_7+\overline{\zeta_7})}}$ (see \cite{Curtis-Reiner}). \\
\\
\textbf{Remark}: Note that, although \cite{Curtis-Reiner} deals with finite groups, the results extend naturally to infinite groups as long as the index of the subgroup is finite.

\paragraph{}We have three $2$-dimensional irreducible representations of $G_{\mathbb{Q}}$; therefore, there are three newforms associated to them \cite[Chapter~8]{Diamond-Shurman}. Comparing the $L$-function of the Hecke character to the product of $L$-functions of three modular forms reveals that the product of levels of newforms is $D(\mathbb{Q}(\zeta_7)/\mathbb{Q})N(1-\zeta_7)=7^6$. Since there are no new forms of level $7$ and weight $2$, we are looking at three for weight-2 CM modular forms of level $49$.  These are modular forms with the CM field $\mathbb{Q}(\sqrt{-7})$. 
  
\paragraph{} Denote by $S$  the set of three newforms of weight $2$, level $49$ and CM field $\mathbb{Q}(\sqrt{-7})$. We claim now and establish in the next section that the $L$-function of the Klein Quartic curve is the product of the $L$-function of the three new forms in the set $S$. In terms of the local Zeta function, this means that the numerator of the local Zeta function of Klein Quartic comes as a product of inverses of the Euler $p$ factor of the associated modular forms, i.e.,   
\begin{equation}\label{MF equation} 
    Z_p(KQ;T)=\frac{\prod_{f\in S}(1-a_p(f)T+\chi_f(p)\cdot pT^2) }{(1-T)(1-pT)}, 
\end{equation} 
where $\chi_f$ can be taken as an order $3$ multiplicative character of conductor $7$ in $(\mathbb{Z}/49\mathbb{Z})^\times$, with $\chi_f(3)=\exp(4\pi\iota/3)$ (LMFDB label $7.c$).  
  
\section{Proof of Theorem \ref{theorem 1}} 
  
We will show that the product of inverses of the Euler $p$ factor of the newforms in S is equal to the numerator of the local Zeta function of the Klein Quartic curve for all unramified primes, i.e., $p\neq 7$.
The product of the $L$-function of newforms in the set $S$ is the $L$-function of the Galois representation, $L_{\mathbb{Q}}(\theta,s)$. Let $K=\mathbb{Q}(\zeta_7)$, then by Langland's base change lifting theorem \cite{Langlands}, we have $L_{\mathbb{Q}}(\theta,s)=L_{K}(\rho,s)$, which is

\begin{align} 
&=\prod_{\substack{\mathfrak{p}\in \mathcal{O}_K \\ \mathfrak{p}\nmid 7}}\frac{1}{1-(N_{K/\mathbb{Q}}(\mathfrak{p}))^{-s}\rho(\text{Frob}_\mathfrak{p})}\\ 
   &=\prod_{\substack{\mathfrak{p}\in \mathcal{O}_K \\ \mathfrak{p}\nmid 7}}\frac{1}{1+(p^r)^{-s} J_{p^r}(\chi_{7,\mathfrak{p}}, \chi_{7,\mathfrak{p}}^2)},   
\end{align} 
with $r$ being the degree of the residue field $\mathbb{Z}[\zeta_7]/\mathfrak{p}$. 
  From this, we can read that the product of inverses of Euler $p$-factors for newforms in $S$, at prime $p$, in the variable $T$ for $p^{-s}=T$ is   
  \begin{equation}\label{Euler product} 
      \prod_{\mathfrak{p}\mid p}  ( 1 +J_{p^r}(\chi_{7,\mathfrak{p}}, \chi_{7,\mathfrak{p}}^2 )T^{r}). 
  \end{equation}

\paragraph{} Let $p\neq7$ be a prime, such that $p$ has order $r$ modulo $7$, i.e. $p^r\equiv1\mod7$, then 
$$N_{p^{r}}(KQ)=p^{r}+1+\sum_{i=1}^6J_{p^{r}}(\chi_7^i,\chi_7^{2i}).$$ 
Suppose $r$ is even, that is, $2$ or $6$. According to lemma \ref{Jacobi sum congruence}, all $6$ Jacobi sums have the same value, precisely $p^{r/2}$; therefore, $\alpha_i$'s are $r$-th roots of $p^{r/2}$. Combining this with the fact that $\sum_{i=1}^6\alpha_i^k=0$ for $k\in\{1,...,r-1\}$ (refer section \ref{section point counting modulo p}), gives  
$$\prod_{i=1}^6(1-\alpha_iT)=(1+J_{p^r}(\chi_7,\chi_7^2)T^{r})^{6/r}=(1+p^{r/2}T^r)^{6/r}.$$ 
This matches the equation \eqref{Euler product}.\\ 
\\ 
For $r=1$, it is already established from the discussions in earlier sections.  For $r=3$, we have $J_{p^3}(\chi_7,\chi_7^2)\neq\overline{J_{p^3}(\chi_7,\chi_7^2)}$. With $3$ out of $6$ $\alpha_i$'s satisfying $\alpha_i^3=J_{p^3}(\chi_7,\chi_7^2)$, and the rest $3$ satisfying $\alpha_i^3=\overline{
J_{p^3}(\chi_7,\chi_7^2)}$. Again, combining it with the fact that $\sum_{i=1}^6\alpha_i^k=0$ for $k=1$ and $2$, gives 
$$\prod_{i=1}^6(1-\alpha_iT)=(1+J_{p^3}(\chi_7,\chi_7^2)T^{3})(1+\overline{J_{p^3}(\chi_7^3,\chi_7^6)}T^{3}),$$
which is equal to the expression in equation \ref{Euler product}.
This completes the proof of the Theorem \ref{theorem 1} and the description of the local Zeta function of the Klein Quartic curve for all the unramified primes in $\mathbb{Q}(\sqrt{-7})$.

\paragraph{} The Hasse-Weil $L$-function of a variety is defined by multiplying $\zeta(s)\zeta(s-1)$ to the product of the inverses of the numerator of its local Zeta function for all primes. Where $\zeta(s)$ is the usual Riemann Zeta function. Theorem \ref{theorem 1} shows that the Hasse-Weil $L$-function of the Klein Quartic curve comes as a product of $L$-functions of the three newforms described above. Therefore, we also note that it has an analytic continuation to the entire complex plane.  
  
\paragraph{}\textit{\bf Acknowledgements.} The author thanks Prof. Ling Long for introducing the problem. Additionally, the author thanks her, Prof.~Fang-Ting Tu, and Emma Lien for suggesting corrections and comments on the earlier drafts of this article. The author also thanks the anonymous referee(s) for their valuable comments, which greatly helped to improve this article. Lastly, the author is also grateful to the College of Science and the Department of Mathematics at Louisiana State University for the financial support of this research during the summer of 2023.

  \nocite{*}
 \bibliographystyle{plain}\bibliography{main}
  
\end{document}